\documentclass[11pt]{amsart}

\usepackage{amssymb,amsmath,amsthm}
\usepackage{latexsym}
\usepackage{arydshln}
\usepackage{graphicx}
\graphicspath{{figures/}}
\usepackage{tikz-cd}
\usepackage{mathtools,amssymb}
\usepackage{kotex}
\usepackage{fancyhdr,geometry}
\usepackage{caption}
\usepackage{xcolor}
\usepackage{enumitem}

\theoremstyle{plain} 
\newtheorem{theorem}{Theorem}[section]
\newtheorem{proposition}[theorem]{Proposition}
\newtheorem{lemma}[theorem]{Lemma}
\newtheorem{corollary}[theorem]{Corollary}

\theoremstyle{definition}
\newtheorem{definition}[theorem]{Definition}
\newtheorem{example}[theorem]{Example}
\newtheorem{conjecture}[theorem]{Conjecture}

\theoremstyle{remark}
\newtheorem{remark}[theorem]{Remark}

\usepackage{tikz,tikz-cd}
\usetikzlibrary{calc}
    \tikzset{vert/.style={circle, draw=black!100,fill=black!100,thick, inner sep=0pt, minimum size=1mm},
                 empty/.style={fill=none, draw=none, minimum size=0mm},
                 square/.style={rectangle, draw=red,  minimum width=4mm, minimum height=4mm, fill=none},
                 over/.style={white, ultra thick, double=black, double distance = .4pt} }
\usetikzlibrary{decorations.pathmorphing}
\usetikzlibrary{shapes}
\tikzset{snake/.style={decorate, decoration=snake}}

\newcommand{\clipbox}[4]{
  \clip (#1,#2) -- (#3,#2) -- (#3,#4) -- (#1,#4) -- (#1,#2);
  \draw[dashed]  (#1,#2) -- (#3,#2); 

}

\newcommand{\diagTop}[2]{\begin{scope}[xshift=#1cm, yshift=#2cm]
    \draw (1,0) -- (2,0); 
    \draw (3,0) -- (4,0); 
    \draw (0,0) arc(180:0:2.5cm and 1cm);
    \end{scope}}

\newcommand{\diagBot}[2]{\begin{scope}[xshift=#1cm, yshift=#2cm]
    \draw (1,0) -- (2,0); 
    \draw (3,0) -- (4,0); 
    \draw (0,0) arc(-180:0:2.5cm and 1cm);
    \end{scope}}

\newcommand{\pCr}[2]{ \begin{scope}[xshift=#1cm, yshift=#2cm]
       \draw (0,1) arc(180:240:.5cm);\draw (1,0) arc(0:60:.5cm); 
       \draw (1,1) arc(0:-90:.5cm); \draw (0,0) arc(180:90:.5cm); 
       \end{scope} }

\newcommand{\pCrr}[2]{ \begin{scope}[xshift=#1cm, yshift=#2cm]
      \draw (0,1) arc(180:270:.5cm);\draw (1,0) arc (0:90:.5cm);
       \end{scope} }

\newcommand{\ol}[3]{
 \draw(#1,{#2+.1}) arc (180:270:.1cm); 
 \draw[over] (1.1,#2) -- (#3,#2);
}

\newcommand{\smcr}[2]{
 \begin{scope}[xshift=#1cm, yshift=#2cm-.5cm]
 \begin{scope}[yshift=.165cm, yscale=2/3 ]
 \draw (0,.5) to[out=0, in=-115] (.30,.75) to[out=65, in=180]  (.5,1) to[out = -15, in =70] (.6,.6) to[out = 35, in = 35-180] (.4,.4) to[out = 70-180, in=-15-180]  (.5,0) to[out=180-180, in=65-180] (.75,.3) to[out = -115-180 ,in = 0-180] (1,.5)  ;
        \draw[red] (.4,.4) -- (.30,.75); 
        \draw[red] (.6,.6) -- (.75,.30); 
  \end{scope}
  \draw[red] (.5,.835) -- (.5,1);
  \draw[red] (.5,.165) -- (.5,0); 
 \end{scope} 
}

\newcommand{\scc}[3]{
\draw(#1,#2) arc (#3*90:(#3*90)-90:.1cm);}

\newcommand{\sqp}[5]{
  \scc{#1-.1}{#2}{1}
  \draw(#1, #2-.1) -- (#1,  #3 +.1);
  \scc{#1}{#3+.1}{4}
  \draw(#1-.1,#3) -- (#4+.1,#3);
  \scc{#4+.1}{#3}{3}
  \draw(#4,#3+.1) -- (#4, #5);
}

\newcommand{\old}[3]{
 \draw(#1,#2-.1) arc (180:90:.1cm); 
 \draw[over] (#1+.1,#2) -- (#3,#2);
}

\newcommand{\pCh}[3]{\foreach \y in {#2,...,#3}{\pCr{#1}{\y};}}  

\newcommand{\nCr}[2]{ \begin{scope}[xshift=#1cm, yshift=#2cm]
       \draw (0,1) arc(180:270:.5cm);\draw (1,0) arc(0:90:.5cm); 
       \draw(1,1) arc(0:-60:.5cm); \draw(0,0) arc(180:120:.5cm); 
       \end{scope} }

\newcommand{\nCh}[3]{\foreach \y in {#2,...,#3}{\nCr{#1}{\y};}}

\newcommand{\nsj}[2]{ \begin{scope}[xshift=#1cm, yshift=#2cm]
       \draw (0,0) -- (0,1); \draw (1,0) -- (1,1); 
       \draw[red] (0,.5) -- (1,.5); 
       \draw[red] (0,.25) edge[bend left = 90] (0,.75);
       \draw[red] (1,.25) edge[bend right = 90] (1,.75);
 
       \end{scope} }

 \newcommand{\Verts}[2][,]{\foreach \i/\x/\y in {#2}{\draw (\x#1\y) node (\i){};}}
 \newcommand{\Edges}[2][black]{\foreach \i/\j in {#2}{\draw (\i) edge[#1] (\j);}}

\begin{document}

\title{On geometric realizations of the extreme Khovanov homology of pretzel links}

\author{Jinseok Oh}
\address{Department of Mathematics, Kyungpook National University, Daegu, 41566, Republic of Korea}
\email[Jinseok Oh]{jinseokoh@knu.ac.kr}

\author{Mark H. Siggers}
\address{Department of Mathematics, Kyungpook National University, Daegu, 41566, Republic of Korea}
\email[Mark H. Siggers]{mhsiggers@knu.ac.kr}

\author{Seung Yeop Yang}
\address{KNU G-LAMP Project Group, KNU Institute of Basic Sciences, Department of Mathematics, Kyungpook National University, Daegu, 41566, Republic of Korea}
\email[Seung Yeop Yang]{seungyeop.yang@knu.ac.kr}

\author{Hongdae Yun}
\address{Department of Mathematics, Kyungpook National University, Daegu, 41566, Republic of Korea}
\email[Hongdae Yun]{yyyj1234@knu.ac.kr}

\subjclass{}%

\begin{abstract}
Gonz\'alez-Meneses, Manch\'on, and Silvero showed that the (hypothetical) extreme Khovanov homology of a link diagram is isomorphic to the reduced (co)homology of the independence simplicial complex of its Lando graph. Przytycki and Silvero conjectured that the extreme Khovanov homology of any link diagram is torsion-free. In this paper, we investigate explicit geometric realizations of the real-extreme Khovanov homology of pretzel links. This gives further support for the conjecture.
\end{abstract}

\keywords{Khovanov homology, pretzel link, geometric realization, homotopy type.}

\subjclass[2020]{Primary: 57K10, 57K18. Secondary: 57M15.}

\maketitle

\section{Introduction}

Khovanov homology \cite{Kho00}, as a categorification of the Jones polynomial, is a powerful invariant of knots and links. Various attempts \cite{Chm05, ET14} have been made to establish geometric realizations of Khovanov homology. Ultimately, Lipshitz and Sarkar \cite{LS14} constructed spectra $\mathcal{X}_{Kh}^{j}(L)$ whose reduced singular cohomology is isomorphic to the Khovanov homology of a link $L.$
Meanwhile, Gonz\'alez-Meneses, Manch\'on, and Silvero \cite{GMS18} established a concrete geometric realization for the (hypothetical) extreme Khovanov homology $KH^{i,j_{\rm min}}(D_L)$ of $L$ by demonstrating an isomorphism between $KH^{i,j_{\rm min}}(D_L)$ and the reduced cohomology of the independence simplicial complex of a special graph obtained from $D_L,$ called the Lando graph. 
It was proven by Cantero Mor\'an and Silvero \cite{CS20} that the spectrum constructed by Gonz\'alez-Meneses, Manch\'on, and Silvero \cite{GMS18} is stably homotopy equivalent to the one introduced by Lipshitz and Sarkar \cite{LS14} at its extreme quantum grading.
Przytycki and Silvero \cite{{PS18,PS20}} extended the results and proposed the following conjecture:

\begin{conjecture} \cite{{PS18}} \label{conjecture}
The independence simplicial complex associated with a circle graph is homotopy equivalent to a wedge of spheres. In particular, the extreme Khovanov homology of any link diagram is torsion-free.
\end{conjecture}

In this paper, we construct explicit geometric realizations of the real-extreme Khovanov homology of pretzel links by searching for suitable link diagrams. Furthermore, we show that the homotopy types of these geometric realizations are wedges of spheres. This is in line with above conjecture.

\section{Khovanov homology}
We review the definition of Khovanov homology based on \cite{Vir04}. Let $D_L$ be a link diagram of an oriented link $L.$ Let $p$ and $n$ be the numbers of positive and negative crossings in $D_L$, respectively. See Figure \ref{overflow}(i). The \emph{writhe} of $D_L$ is $\omega(=\omega_{D_L})=p-n.$
\begin{figure}[ht!]
  \centering
  \includegraphics[width=90mm]{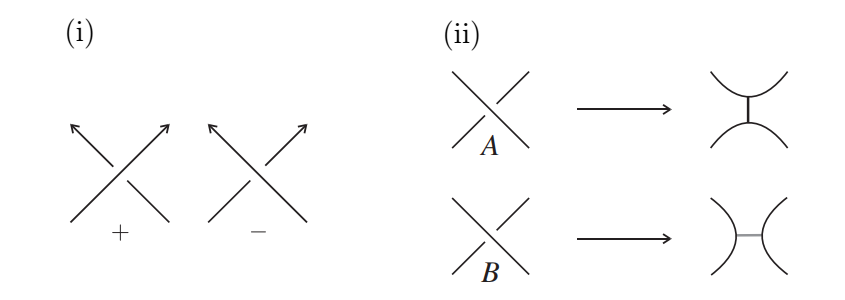}
  \caption{(i) Convention of positive and negative signs (ii) The smoothing of a crossing according to an $A$-label and a $B$-label }\label{overflow}
\end{figure}
Let $Cr(D_L)$ be the set of all crossings of $D_L$.
A \emph{Kauffman state} of $D_L$ is a function $s : Cr(D_L) \rightarrow \{A,B\}$.
Note that the collection $\mathcal{S}$ of all possible Kauffman states of $D_L $ has $2^c$ elements, where $c$ is the number of crossings of $D_L.$
Let $\sigma(s)$ be the difference between the number of $A$-labels and $B$-labels of a Kauffman state $s,$ i.e., $\sigma(s) = \mid s^{-1}(A)\mid$ $-$ $\mid s^{-1}(B)\mid.$ 
Let $D_s$ be the result of smoothing $D_L$ by assigning either an $A-$ or $B-$smoothing to each crossing according to the label of a Kauffman state $s$ as depicted in Figure \ref{overflow}(ii). We denote the number of circles of $D_s$ by $|s|.$

An {\it enhanced Kauffman state} $S$ of an oriented link diagram $D_L$ is a Kauffman state $s$ together with a map which associates a sign $\varepsilon_i = \pm 1$ to each of the circles of $D_s.$ We let $\tau = \sum_{i=1}^{\mid s \mid} \varepsilon_i$.
For an enhanced Kauffman state $S,$ we define two integers by
$$i = i(S) = \cfrac {\omega-\sigma}{2},~~ j = j(S) = \omega+i+\tau.$$

Let $S$ and $T$ be enhanced states of an oriented link diagram $D_L.$
The states $T$ and $S$ are {\it adjacent} if they satisfy the following conditions:

\begin{enumerate}
\item[\rm(1)]
$i(T)=i(S)+1$ and $j(T)=j(S).$
\item[\rm(2)]
$S$ and $T$ associate identical labels to all crossings except one denoted by $x=x(S,T)$, where
$S$ assigns an $A$-label and $T$ a $B$-label.
\item[\rm(3)]
$S$ and $T$ assign the same signs to the common circles in $D_S$ and $D_T$.
\end{enumerate}

Note that the circles which are not common in $D_S$ and $D_T$ are those
touching the crossing $x.$
Figure \ref{(possi)} shows the different possibilities of going from $D_S$ to $D_T,$ where $T$ is adjacent to $S.$

\begin{figure}[h]
\centering
\includegraphics[width=120mm]{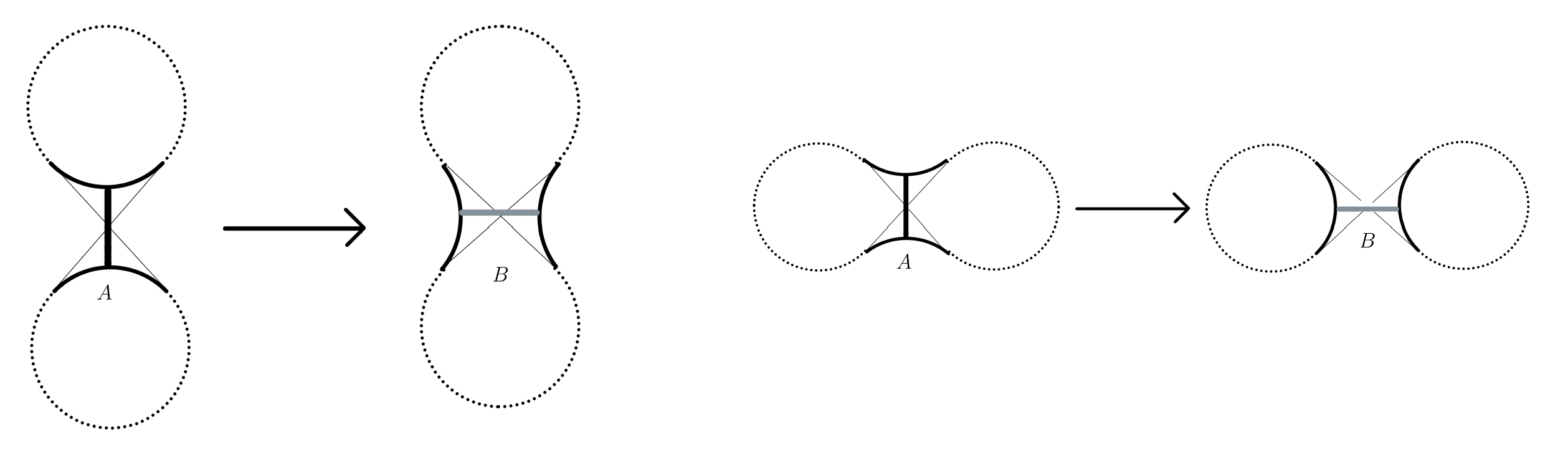}
\caption{All possible enhancements when melting two circles are
$(++ \longrightarrow +)$, $(+-\longrightarrow-)$,$(-+\longrightarrow-)$.
The possibilities for the splitting are
$(-\longrightarrow--)$,$(+\longrightarrow+-)$,$(+\longrightarrow-+)$.}\label{(possi)}
\end{figure}

Let $C^{i,j}(D_L)$ be the free abelian group generated by the set of enhanced states $S$ of $D_L$ with $i=i(S)$ and $j=j(S)$.
We order the crossings in $D_L.$
For each fixed integer $j,$ let us consider the ascendant complex
$$
\cdots \rightarrow C^{i-1,j}(D_{L}) \xrightarrow{\partial_{i-1}} C^{i,j}(D_{L}) \xrightarrow{\partial_{i}} C^{i+1,j}(D_{L}) \rightarrow \cdots
$$
together with boundary operations $\partial_i (S) = \sum (S : T)T,$ where
\begin{equation*}
(S : T)=
\begin{cases}
0 & \text{if } T \text{ is not adjacent to} ~S; \\
(-1)^k & \text{otherwise}.
\end{cases}
\end{equation*}
Here, $k$ is the number of $B$-labeled crossings coming after the changed crossing.
It turns out that $\partial_{i} \circ \partial_{i-1}=0$, and therefore
$\{C^{i,j}(D_L),\partial_i\}$ forms a chain complex.
Khovanov showed that the homology groups yielded from the above chain complex are independent on the choice of link diagrams, i.e., these homology groups are link invariants. See \cite{Kho00} for further detatils.

\begin{theorem}
The family of (co)homology groups
\begin{center}
$KH^{i,j}(L) = KH^{i,j}(D_L) =\cfrac{\textrm{Ker}{(\partial_i)}}{\textrm{Im}{(\partial_{i-1})}}$
\end{center}
are called the \emph{Khovanov (co)homology} of an oriented link $L.$
\end{theorem}

Let $j_{\rm min}=j_{\rm min}(D_L)=
{\rm min}\{j(S)~|~ S \text{~  is an enhanced state of} ~D_L\}.$
The chain complex $\{C^{i,j_{\rm min}}, \partial_{i}\}$ is 
the {\it extreme Khovanov chain complex}, and the corresponding homology
groups $KH^{i,j_{\rm min}}(D_L)$ are the {\it (potential) extreme Khovanov homology groups}.

\begin{corollary} \emph{{\cite{GMS18}}} \label{GMS cor}
Fix an oriented link diagram $D$ with $c$ crossings and $n$ negative crossings. Then $j_{\rm min}=c-3n-|s_{A}D|,$ where $|s_{A}D|$ is the number of circles of $s_{A}D.$
\end{corollary}

We remark that the integer $j_{\rm min}$ depends on the link diagram $D_L.$
That is, given two link diagrams $D_L$ and $D'_L$ of a link $L,$
$j_{\rm min}(D_L)$ and $j_{\rm min}(D'_L)$ may differ.
However, the smallest value of $j$ of $L,$ denoted by $\underline{j},$ such that
$KH^{i,\underline{j}}(L)$
is non-trivial for some $i$ is unique
because this value does not depend on the choice of link diagrams of $L.$

\begin{definition}
The homology groups $KH^{i,\underline{j}}(L)$ are called the {\it real-extreme Khovanov homology groups} of an oriented link $L.$
\end{definition}

Note that $j_{\rm min}(D_L) \leq \underline{j}(L)$ for every link diagram $D_L$ of $L.$

\section{Independence simplicial complex and their homotopy type}

This section explores how, by using the result that the extreme Khovanov complex of a link diagram can be expressed as the independence complex of its Lando graph, one can construct an explicit geometric realization of the extreme Khovanov homology. Additionally, we recall several technical methods useful for studying independence complexes of graphs.

\begin{definition}
Let ${V} = \{v_0, \dots, v_n\}$ be a set of ~$n+1$ symbols.
An {\it abstract simplicial complex} $\mathcal{K}=(V(\mathcal{K}), P(\mathcal{K}))$ on ${V}$ consists of $V(\mathcal{K})=V$ and $P(\mathcal{K}) \subset 2^{V}$ if $P(\mathcal{K})$ is satisfying the following condition:
\begin{enumerate}
\item If $\sigma \in P(\mathcal{K})$ and $\tau \subset \sigma$,
then $\tau \in P(\mathcal{K})$.
\item $\{v_i\} \in P(\mathcal{K})$ for every $v_i \in V(\mathcal{K})$.
\end{enumerate}
${V}(\mathcal{K})$ and ${P}(\mathcal{K})$ are called the \emph{set of vertices} and the \emph{collection of simplices} of $\mathcal{K},$ respectively.
\end{definition}

\begin{definition}

Let $\mathcal{K}_1$ and $\mathcal{K}_2 $ be two simplicial complexes.
\begin{enumerate}
\item For each $i = 1, 2,$ we choose a distinguished $0$-simplex $v_i$ in $\mathcal{K}_i$. The \emph{wedge product} of $\mathcal{K}_1$ and $\mathcal{K}_2$ is the simplicial complex obtained by identifying $v_1$ and $v_2$. It is denoted by $ (\mathcal{K}_1, v_1) \vee (\mathcal{K}_2, v_2) .$
\item The \emph{join} of $\mathcal{K}_1$ and $\mathcal{K}_2 $, denoted by $ \mathcal{K}_1 * \mathcal{K}_2 $, is defined as a simplicial complex on the set $ V(\mathcal{K}_1) \cup V(\mathcal{K}_2) $ whose simplices are disjoint union of the simplices of $\mathcal{K}_1$ and of $\mathcal{K}_2.$
\item The \emph{suspension} of $\mathcal{K}_1$, denoted by $ \Sigma \mathcal{K}_1 $, is the join $ \mathcal{K}_1 * S^0 $, where $S^0$ is a discrete space with two points.
\end{enumerate}

\end{definition}

Let us review how to construct the Lando graph from a given link diagram and its associated independence complex.

\begin{definition} \cite{{BM00, Man04}}
Let $D$ be a link diagram, and let ${{s}_{A}}D$ be the state assigning $A$-labels to all the crossings of $D$. An A-chord is said to be
{\it admissible} if it connects the same circle of ${{s}_{A}}D$ to itself.
Then the {\it Lando graph} of link diagram $D,$ denoted by
${{G}_{D}},$ is constructed from ${{s}_{A}}D$
by considering a vertex for every admissible $A$-chord,
and an edge joining two vertices if corresponding $A$-chords alternate
in the same circle.
\end{definition}

\begin{definition}
Let $G$ be an undirected graph, and let $V(G)$ be the set of vertices of $G.$
The abstract simplicial complex $I(G)=(V(I(G)), P(I(G)))$ such that $V(I(G))=V(G)$ and $P(I(G))=\{ \sigma  \mid \sigma \text{~is an independent subset of~} V(G) \}$ is called the {\it independence complex} of $G$.
\end{definition}

\begin{example}
{\rm
Consider the link diagram of the Hopf link shown in Figure \ref{Hopf}(i).
Then ${{s}_{A}}D$ is as depicted in Figure \ref{Hopf}(ii), and so
the corresponding Lando graph is ${{G}_{D}}$ shown in Figure \ref{Hopf}(iii).
}
\end{example}

\begin{figure}[h]
\centering
\includegraphics[width=110mm]{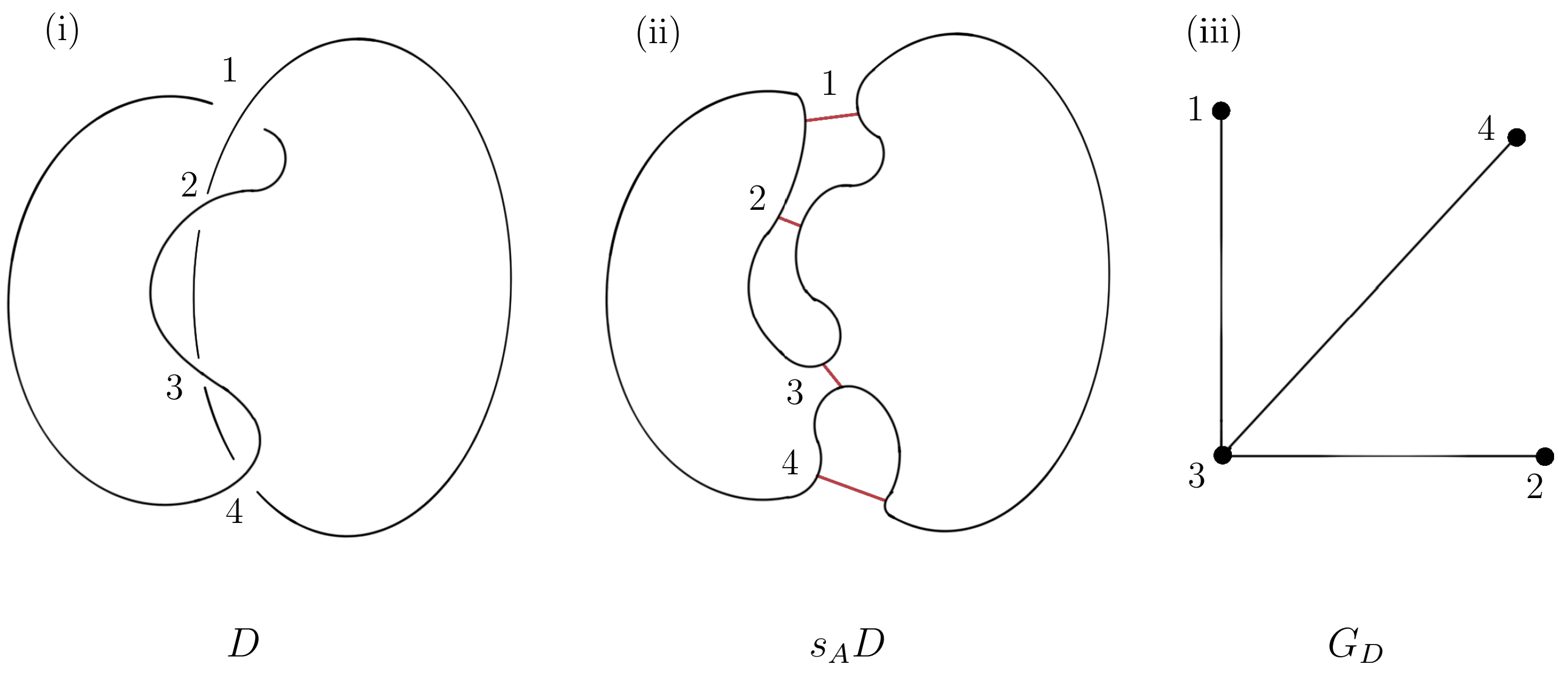}
\caption{The process of obtaining Lando graph from a given link diagram }\label{Hopf}
\end{figure}

\begin{definition} \cite{{GMS18}}
Let $D$ be a link diagram, and let $G_D$ be the Lando graph of $D.$ Let $C^i(I({G_D}))$ be the free abelian group generated by $i$-dimensional
simplices of the associated independence complex $I({G_D}).$ Consider the standard differentials $\delta_{i}$  
\[
\cdots \rightarrow C_{i-1}(I({G_D})) \xrightarrow{\delta_{i-1}} C_{i}(I({G_D})) \xrightarrow{\delta_{i}} C_{i+1}(I({G_D})) \rightarrow \cdots,
\]
that is, $\delta(\sigma) = \sum_v (-1)^k (\sigma \cup {v})$,
where $v$ ranges over the set of vertices of $G_D$ and $k$ is the number of vertices in $\sigma$ that follow $v$ in the predetermined order of the vertices of $G_D$.
Then we have $\delta_{i}\circ\delta_{i-1}=0$, and so $\{C^i(I({G_D))}, \delta_i \}$ forms a chain complex, called the \emph{Lando ascendant complex} of $D.$
The reduced cohomology groups of $\{C^i(I({G_D))}, \delta_i \}$ are called the \emph{Lando cohomology groups} of $D.$
\begin{center}
${\widetilde{H}}^i(I({G_D})) = \cfrac{\textrm{Ker}(\delta_i)}{\textrm{Im}(\delta_{i-1})}$.
\end{center}

\end{definition}

Gonz\'alez-Meneses, Manch\'on, and Silvero showed that the extreme Khovanov complex of a link diagram can be expressed as the independence complex of its associated Lando graph.

\begin{theorem} \emph{{\cite{GMS18}}} \label{GMS main}
Let $D_L$ be a link diagram of oriented link $L$ having $n$ negative crossings.
Let ${{G}_{D_L}}$ be the Lando graph of $D_L$, and let
${{I}({G_{D_L}})}$ be the independence complex of $G_{D_L}$.
Then we have
\begin{center}
$KH^{i,j_{\rm min}}(D_L) \cong {\widetilde{H}}^{i-1+n}({{I}({G_{D_L}})})$.
\end{center}
\end{theorem}
\vspace{3mm}

To facilitate deeper analysis, we review the homotopy types of independence complexes associated with path graphs, cycle graphs, and subsequently summarize the crucial theorems and properties integral to independence complex theory. See \cite{Koz99, Jon08, PS18} for further details.

We denote the \emph{$n$-path graph} depicted in Figure \ref{(pathgon)}(i) by $L_n$ and the \emph{cycle graph} of order $n$ depicted in Figure \ref{(pathgon)}(ii) by $C_n$, respectively. The homotopy types of the independence complexes of $L_n$ and $C_n$ are as follows.

\begin{figure}[h]
\centering
\includegraphics[width=95mm]{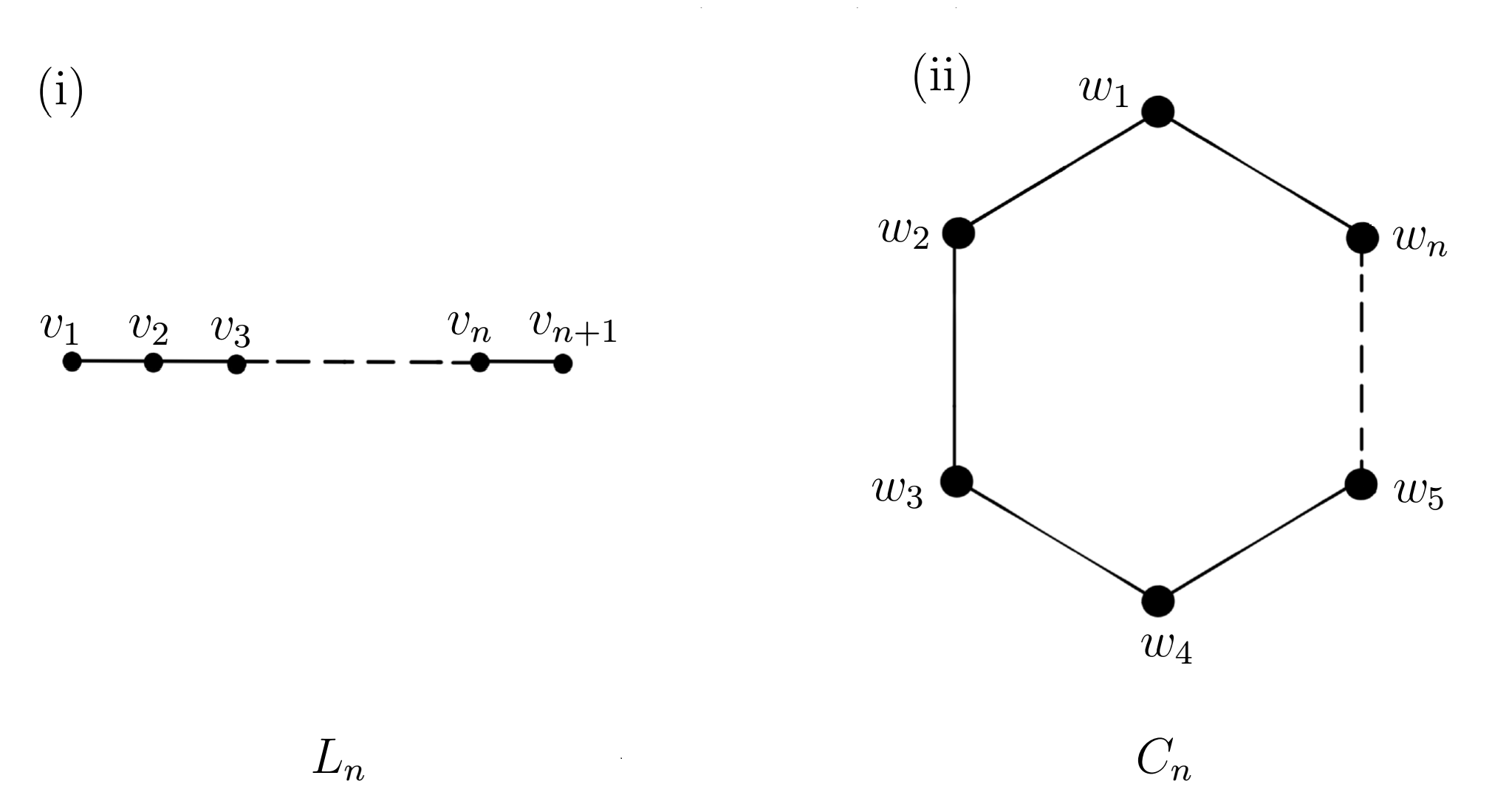}
\caption{The $n$-path graph $L_n$ and the cycle graph $C_n$ of order $n$}\label{(pathgon)}
\end{figure}

\begin{proposition} \emph{{\cite{Koz99}}} \label{(basic)}
For a positive integer $k,$ we have
\begin{enumerate}
    \item
${I}({L_n}) \simeq
\begin{cases}
S^{k-1} & \text{if } n = 3k-2, 3k-1; \\
\emph{a point} & \text{if } n = 3k;
\end{cases}
$
    \item
${I}({C_n}) \simeq
\begin{cases}
S^{k-1} & \text{if } n = 3k \pm 1; \\
S^{k-1} \vee S^{k-1} & \text{if } n = 3k.
\end{cases}
$

\end{enumerate}
\end{proposition}

\begin{remark} \label{(disunion)}
Given two graphs $G_1$ and $G_2,$ the independence complex of the disjoint union $G_1 \sqcup G_2$ of $G_1$ and $G_2$ is homotopy equivalent to the join of their respective independence complexes, that is,
\begin{center}
$I({G_1 \sqcup G_2}) = I({G_1}) * I({G_2}).$
\end{center}
\end{remark}

Let us consider the following two examples to be used in the proofs of Theorem \ref{main1} and Theorem \ref{main2}.

\begin{example} \label{(star)}
Let $R_n$ denote the star graph of $n$ rays of length $2$ as depicted in Figure \ref{(starf)}. Note that the homotopy type of $I({R_n}-{u})$ is $S^{n-1}.$ Since the simplices $[u,w_{1},\dots,w_{n}]$ and
$[w_{1},\dots,w_{n}]$ have the same homotopy type in $I(R_n),$ $I({R_n}) \simeq I({R_n}-{u}).$
Note that the homotopy type of $I({R_n}-{u})$ is equivalent to the join of $n$ copies of $S^0.$ Thus, $I({R_n}) \simeq I({R_n}-{u}) \simeq S^{n-1}.$
\end{example}

\begin{figure}[h]
\centering
\includegraphics[width=75mm]{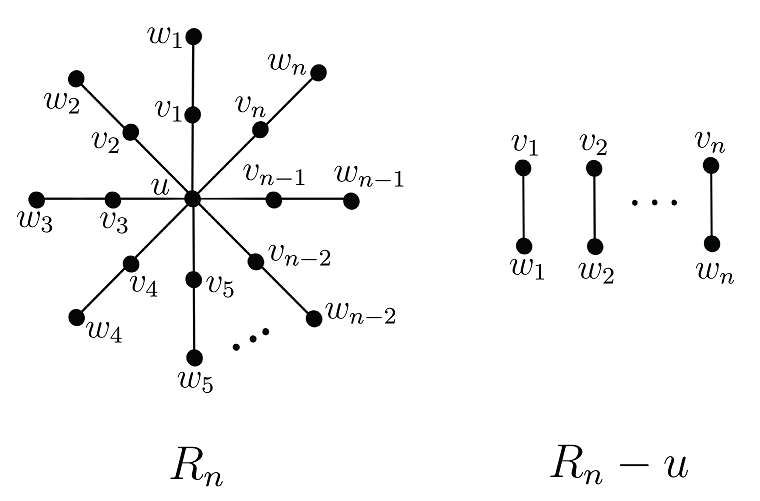}
\caption{The star graph $R_n$ of $n$ rays of length $2$ and $R_n-u$ } \label{(starf)}
\end{figure}

The following are several technical methods that are useful for determining the independence complex of a given graph.

\begin{proposition} \label{(divid)}
\begin{enumerate}
    \item
\emph{{\cite{Cso09}}} {\rm (Csorba reduction)}
Let $H$ be a graph, and let $[u,v]$ be an edge of $H.$ Let $G$ be the graph obtained from $H$ by replacing the edge $[u,v]$ by a path of length $4.$ Then
$I(G) \simeq \sum I(H).$
    \item
\emph{{\cite{PS18}}} {\rm (Generalized Csorba reduction)}
Let $v$ and $w$ be two vertices of a loopless graph $G.$ If $v$ and $w$ are connected by a path $L$ of length three (the two vertices between them having order $2$) and the graph $H$ is obtained from $G$ by contracting $L,$ then $I(G) \simeq \sum I(H).$

\end{enumerate}
\end{proposition}

\begin{corollary}\emph{{\cite{PS18}}}   \label{(glue)}
Let $G$ be a loopless graph. Then
\begin{center}
${I}({G \mid^{1} C_n}) \simeq \sum{}^{k} I(G)$ ~~if $n=3k+2,$
\end{center}
where $G \mid^{1} C_n$ is obtained from $G$ by gluing $C_n$ along an outer edge (See Figure \ref{(gluing)} as an example.).
\end{corollary}

\begin{figure}[h]
\centering
\includegraphics[width=120mm]{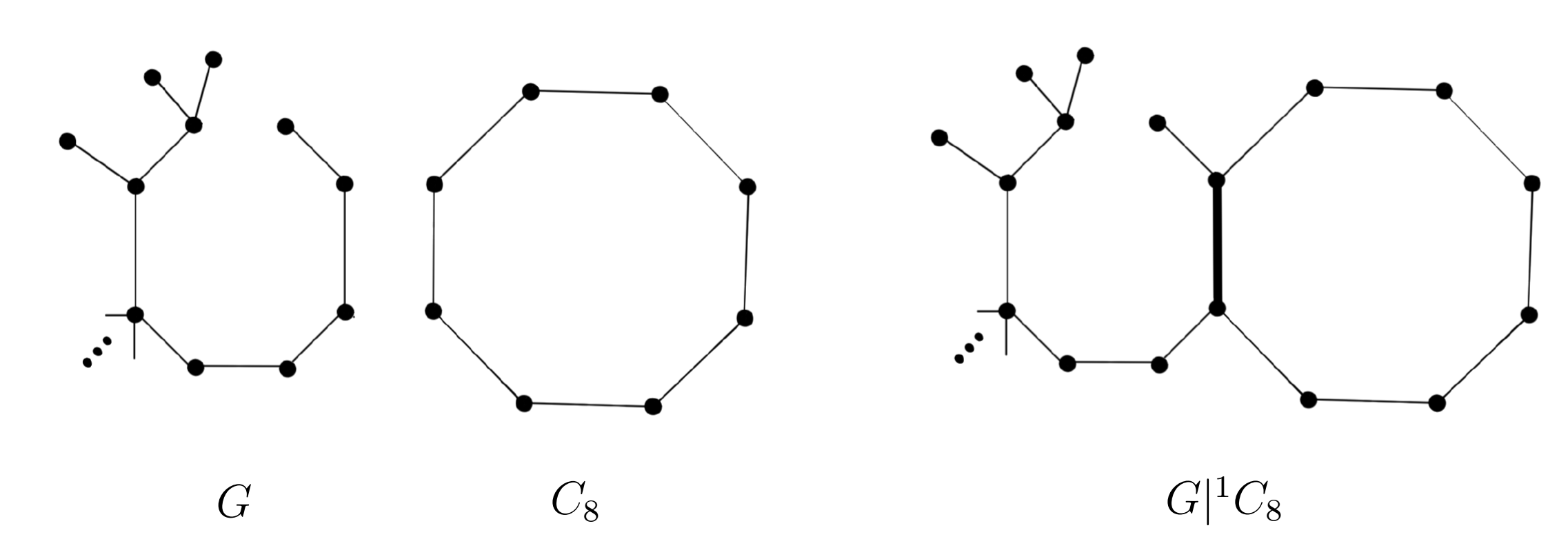}
\caption{The graphs $G,C_{8},$ and $G|^{1}C_{8}.$
The ``gluing edges" have been thickened.
}\label{(gluing)}
\end{figure}

\begin{definition}
For two vertices $v$ and $w$ in a graph $G$, we say that $v$
\emph{dominates} $w$ if ${N_G}(w) \cup \{w\} \cup \{v\} \subset {N_G}(v) \cup \{v\} \cup \{w\}$, where ${N_G}(*)$ is the set of adjacent vertices
to the vertex $*.$
\end{definition}

\begin{lemma} \label{(DL)}
Let $G$ be a graph. Let $v,w$ be two vertices of $G$ such that $v$ dominates $w$.
\begin{enumerate}
\item[\rm(1)] \emph{{\cite{Cso09, Eng09}}}
If \( v \) and \( w \) are not adjacent in \( G \), then $I(G)$ is homotopy equivalent to $I(G - v)$.
\item[\rm(2)] \emph{{\cite{PS18}}}
If \( v \) and \( w \) are adjacent in \( G \), then $I(G)$
is homotopy equivalent to $I(G - v) \vee \Sigma(I(G - \text{st}(v))),$
where ${st}(v)={N_G}(v) \cup \{v\} \cup \{ [u,v] \mid u \in {N_G}(v) \}.$
\end{enumerate}
\end{lemma}

\section{Geometric realizations of the real-extreme Khovanov homology of pretzel links}

In this section, we investigate explicit geometric realizations of the real-extreme Khovanov homology of pretzel links. Recall that a pretzel link $P(\pm p,\pm q,\pm r)$ is a link admitting a link diagram as depicted in Figure \ref{(pretzel)}, where each positive integer $p, q, r$ represents the number of half-twists and each sign determines the type of crossing (positive or negative) in the corresponding box.

\begin{figure}[ht!]
\centering
\includegraphics[width=100mm]{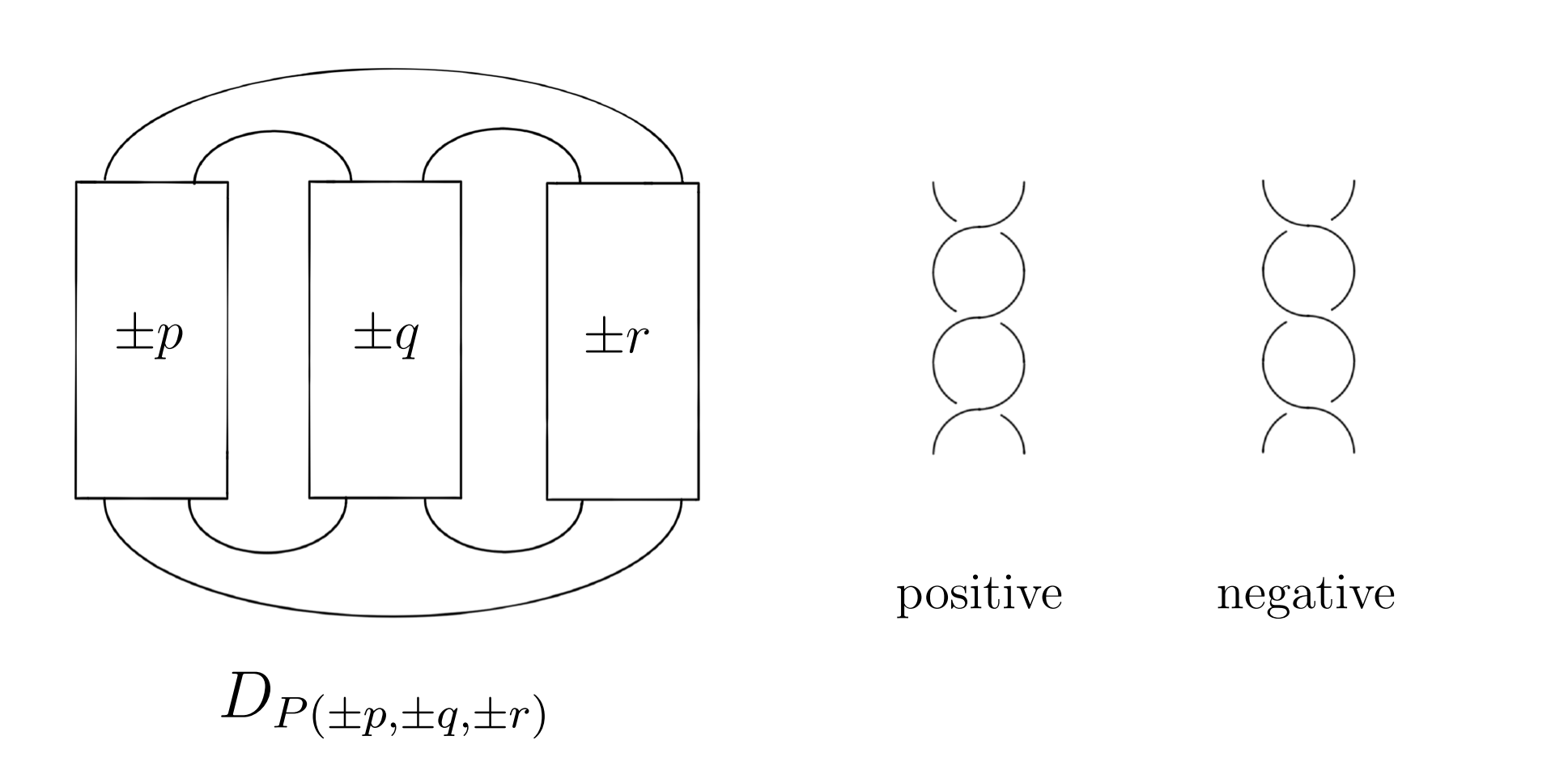}
\caption{A standard link diagram of a pretzel link $P(\pm p,\pm q,\pm r)$}\label{(pretzel)}
\end{figure}


To construct a geometric realization of the real-extreme Khovanov homology of a given link $L$ via Gonz\'alez-Meneses, Manch\'on, and Silvero's method \cite{GMS18}, it is necessary to find a link diagram $D_L$ such that the associated independence complex $I(G_{D_L})$ is non-contractible, where $G_{D_L}$ is the Lando graph of $D_L.$

In the case of the pretzel links $P(p, q, r)$ and $P(-p, -q, -r),$ these are alternating links, and the Lando graphs obtained from their ordinary link diagrams are empty in both cases, meaning that their independence complexes have the homotopy type of $S^{-1}.$ Therefore, we have

\begin{proposition} \emph{{\cite{GMS18}}} \label{main0}
There exist link diagrams $D_{P(p, q, r)}$ and $D_{P(-p,-q,-r)}$ of the pretzel links $P(p, q, r)$ and $P(-p,-q,-r)$, respectively, such that their associated independence complexes $I(G_{D_{P(p, q, r)}})$ and $I(G_{D_{P(-p,-q,-r)}})$ are non-contractible.\\ More precisely, $I(G_{D_{P(p, q, r)}}) \simeq S^{-1} \simeq I(G_{D_{P(-p, -q, -r)}}).$
\end{proposition}

On the other hand, in the case of the pretzel links $P(p, q, -r)$ and $P(p, -q, -r),$ the associated independence complexes obtained from their usual link diagrams are contractible, in general. Therefore, in this case, it is necessary to find appropriate link diagrams to construct geometric realizations of their real-extreme Khovanov homology.

\begin{theorem} \label{main1}
There exists a link diagram $\widetilde{D}_{P(p, q, -r)}$ of the pretzel links $P(p, q, -r)$ such that its associated independence complex $I(G_{\widetilde{D}_{P(p, q, -r)}})$ is non-contractible.\\
More precisely, $I(G_{\widetilde{D}_{P(p, q, -r)}}) \simeq S^{r-1}$.
\end{theorem}

\begin{proof}

The link diagram $\widetilde{D}_{P(p,q,-r)}$ can be deformed from $D_{P(p,q,-r)}$ by the second Reidemeister move to the bottom of the first strand and
each crossing on the third strand except one depicted as Figure \ref{(chamix)}.
Note that the Lando graph associated with $\widetilde{D}_{P(p,q,-r)}$ is $G_{\widetilde{D}_{P(p,q,-r)}}$ in Figure \ref{(Landomix)}(i).
Then $I({G_{\widetilde{D}_{P(p,q,r)}}})$ is homotopy equivalent to
$I({G^{\prime}_{\widetilde{D}_{P(p,q,r)}}})$ as illustrated in Figure \ref{(Landomix)}(ii) by iterated use of Lemma \ref{(DL)}.
Moreover, the homotopy type of independence complex of $G^{\prime}_{\widetilde{D}_{P(p,q,r)}}$ is $S^{r-1}$ by Proposition \ref{(divid)} and Example \ref{(star)}.
Therefore, $I({G_{\widetilde{D}_{P(p,q,r)}}}) \simeq S^{r-1}.$
\end{proof}

\begin{figure}[ht!]
\centering
\includegraphics[width=150mm]{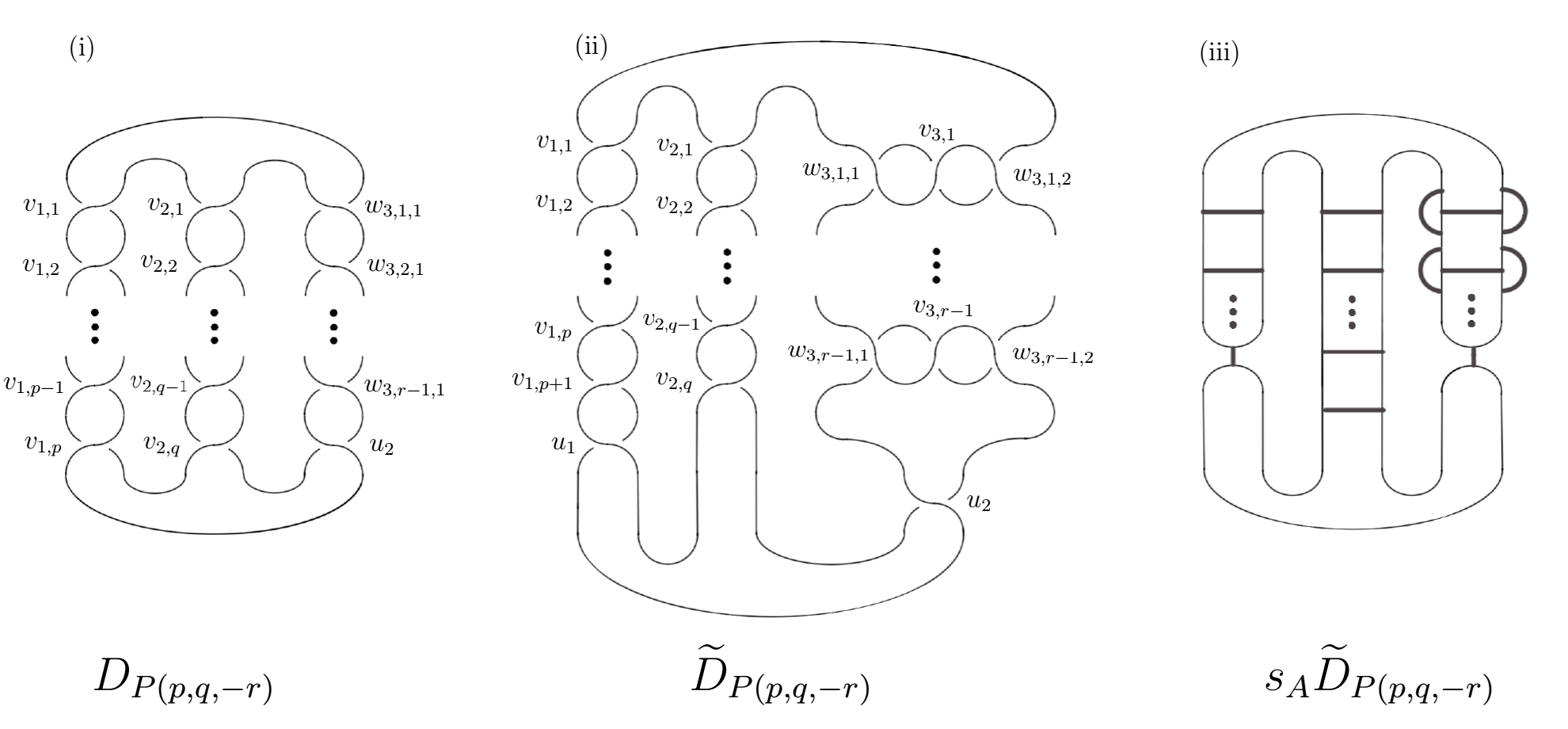}
\caption {The link diagrams ${D}_{P(p,q,-r)} $ and $ \widetilde{D}_{P(p,q,-r)},$
and $s_{A}\widetilde{D}_{P(p,q,-r)}$}\label{(chamix)}
\end{figure}

\begin{figure}[ht!]
\centering
\includegraphics[width=70mm]{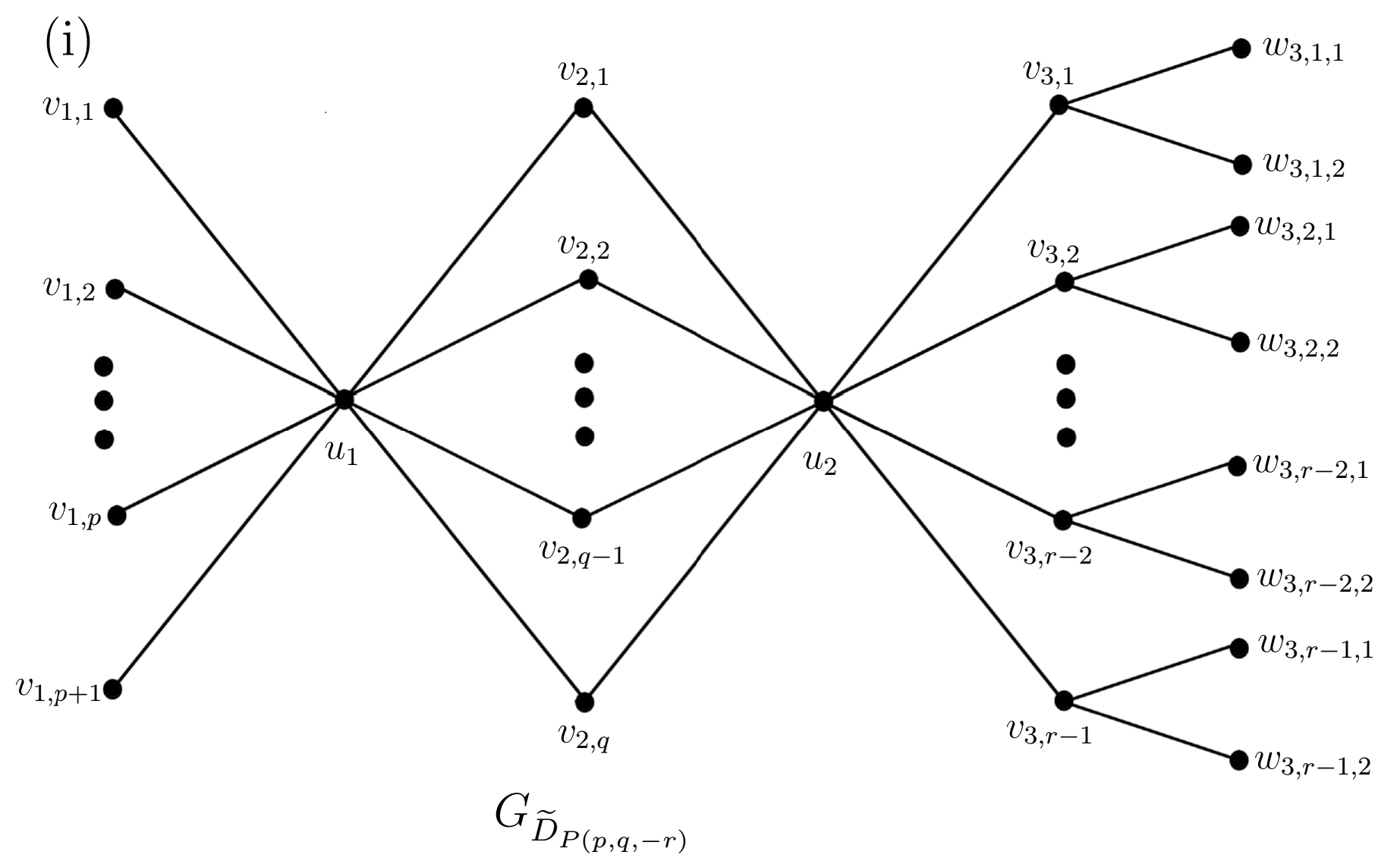}
\includegraphics[width=70mm]{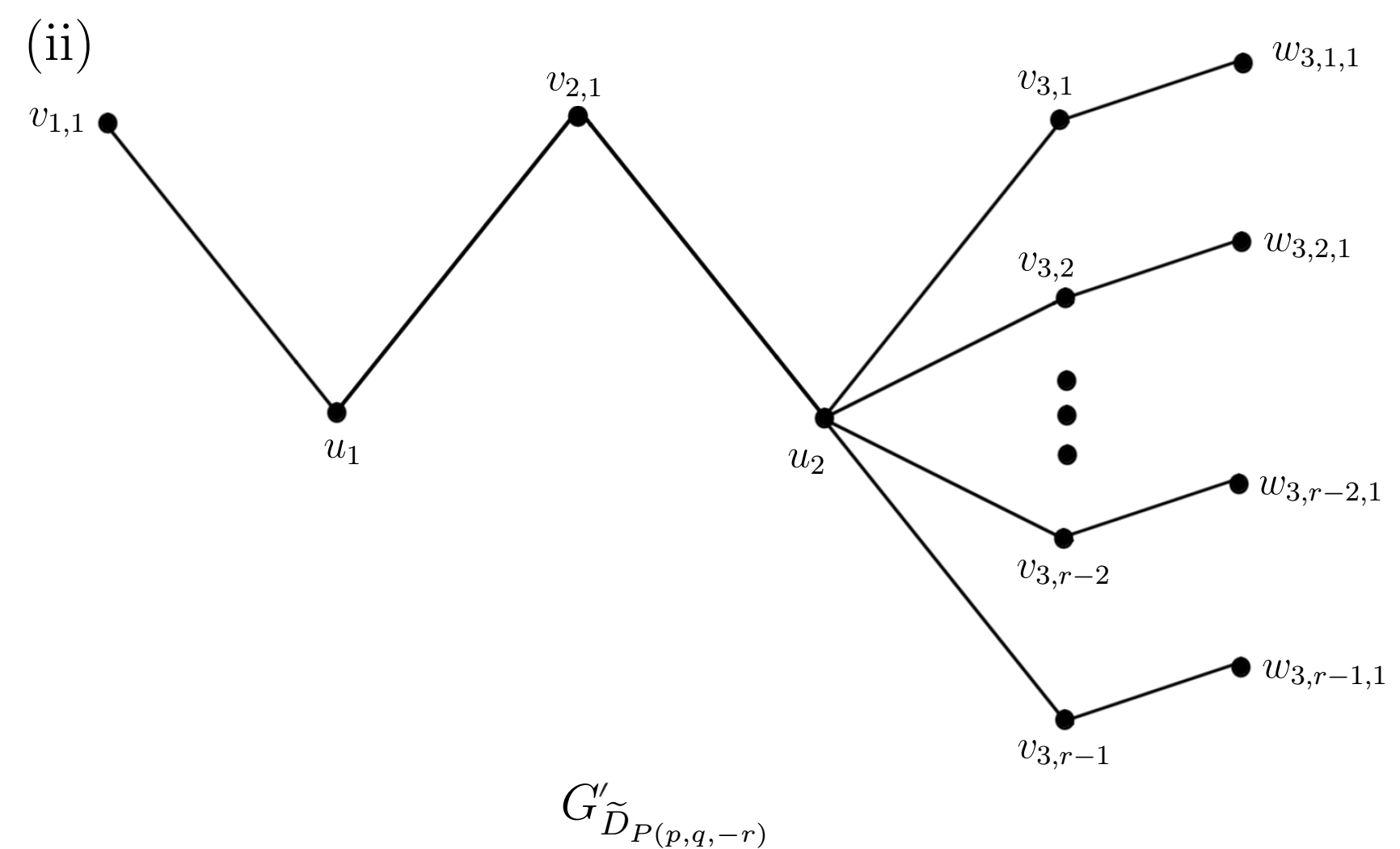}
\caption {The Lando graph of $\widetilde{D}_{P(p,q,-r)}$ and its subgraph $G^{\prime}_{\widetilde{D}_{P(p,q,r)}}$}\label{(Landomix)}
\end{figure}

In the following theorem, for we find link diagrams $\widetilde{D}_P$ of the pretzel link $P(p, -q, -r)$, for most values of $(p,q,r)$, such that its associated independence complex $I(G_{\widetilde{D}_{P}})$ is non-contractible.  We are unable to deal with the case that  $\min\{q,r\} = p+1$ and $q \neq r$. In Remark \ref{rem}, we explain where our construction fails for these  cases.  

\begin{theorem} \label{main2}
Let $(p,q,r)$ be a triple of positive integers such that if $\min\{q,r\} = p+1,$ then $q = r$.  There exists a link diagram $\widetilde{D}_{P}$ of the pretzel links $P(p, -q, -r)$ such that its associated independence complex $I(G_{\widetilde{D}_{P}})$ is non-contractible. \\
More precisely,
\begin{equation*}
I({G_{\widetilde{D}_{P(p,-q,-r)}}}) \simeq
\begin{cases}
(1)~ S^{q+r-2} \vee S^{q+r-2} & \text{if } p \geq q=r;\\
(2)~ S^{q+r-2} & \text{if } p+1=q=r;\\
(3)~ S^{q+r-2} & \text{if } q \leq p {\rm~and~} q<r;\\
(4)~ S^{q+r-2} & \text{if } r \leq p {\rm~and~} r<q;\\
(5)~ S^{q+r-3}& \text{if } p<{\rm min}\{q,r\}-1.\\
\end{cases}
\end{equation*}

\end{theorem}

\begin{proof}
When $p,-q,$ and $-r$ are given, let $P$ be the pretzel link $P(p,-q,-r)$ and $D_P$ be its standard link diagram as given in Figure \ref{(pretzel)}. We will deform $D_P$ to an equivalent link diagram $\widetilde{D}_P$ via Reidermeister moves, and from the smoothed diagram $s_{A}\widetilde{D}_P$ of  $\widetilde{D}_P$, get the Lando graph $G_{\widetilde{D}_P}$ whose independence complex $I(G_{\widetilde{D}_P})$ will have homotopy type as given in the statement of the theorem.          
  
We first consider the base case of $p = q = r$ in some detail, and then explain how the treatment differs from this in the other cases. 
Figure \ref{A1} shows how in the base case the link diagram $D_{P}$ can be deformed by Reidemeister moves to $\widetilde{D}_{P}$. Only the case $p = 7$ is shown, but it should be clear how to continue this: starting at the bottom, we undo the bottom crossing in the leftmost strand of $D_{P}$ by crossing it over the other two strands.  We then undo the next crossing, wrapping it carefully around the previous crossing. Proceeding as shown in the figure, we undo $p-1$ crossings on the leftmost strand, and the resulting diagram is wrapped $p-1$ times in concentric circles around the bottom of the link diagram.

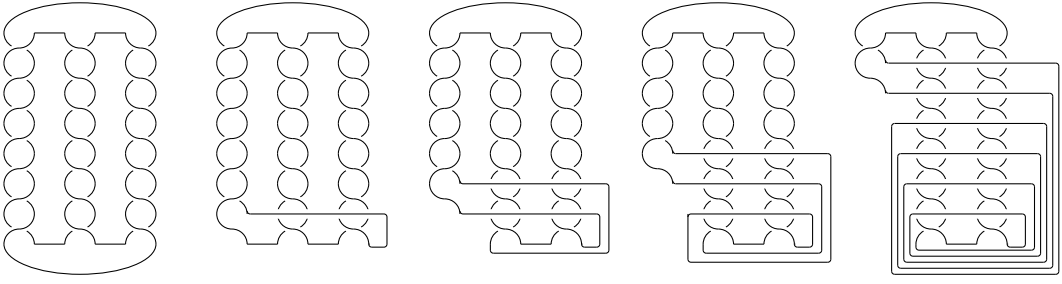
\begin{figure}[ht!]
\begin{tikzpicture}[every node/.style={vert}, scale = .4]
    \diagTop{0}{7}
   \pCh006\nCh206\nCh406
    \diagBot{0}{0}

   \begin{scope}[xshift = 7cm]
    \diagTop{0}{7}
    \pCrr00\pCh{0}{1}{6}
    \nCh206
    \nCh406
    \draw (1,0) -- (2,0); \draw (3,0) -- (4,0);
    \ol{1}{1}{5.5}
    \sqp{5.6}{1}{-.1}{5}{0}
   \end{scope}

   \begin{scope}[xshift = 14cm]
    \diagTop{0}{7}
    \pCrr01\pCh{0}{2}{6}
    \nCh206\nCh406
    \draw (3,0) -- (4,0);
    \ol{1}{1}{5.5}
    \sqp{5.6}{1}{-.1}{5}{0}
    \ol{1}{2}{5.8}
    \sqp{5.9}{2}{-.3}{2}{0} 
   \end{scope}

   \begin{scope}[xshift = 21cm]
    \diagTop{0}{7}
    \pCrr02\pCh036
    \nCh206\nCh406
    \draw (3,0) -- (4,0);
    \old{1.5}{1}{5.5} 
    \sqp{5.6}{1}{-.1}{5}{0}

    \ol{1}{2}{5.8}
    \sqp{5.9}{2}{-.3}{2}{0} 

    \ol{1}{3}{6.1}
    \sqp{6.2}{3}{-.6}{1.5}{1}
   \end{scope}

   \begin{scope}[xshift = 28cm]
    \diagTop{0}{7}
    \pCrr05\pCh066
    \nCh206\nCh406
    \draw (3,0) -- (4,0);
    \old{1.8}{1}{5.5} 
    \sqp{5.6}{1}{-.1}{5}{0}

    \old{1.6}{2}{5.8}
    \sqp{5.9}{2}{-.2}{2}{0} 

    \old{1.4}{3}{6}
    \sqp{6.1}{3}{-.4}{1.8}{.9}

    \old{1.2}{4}{6.2}
    \sqp{6.3}{4}{-.6}{1.6}{1.9}

    \ol{1}{5}{6.4}
    \sqp{6.5}{5}{-.8}{1.4}{2.9}

    \ol{1}{6}{6.6}
    \sqp{6.7}{6}{-1}{1.2}{3.9}
   \end{scope}
   \end{tikzpicture}
   \caption{ The link diagram $\widetilde{D}_{P}$ constructed from 
              $D_{P} = D_{P_{(p,-p,-p)}}$ via Reidermeister moves }\label{A1}
   \end{figure}

   \begin{figure}
   \begin{tikzpicture}

   \begin{scope}[xshift=0cm, scale=.5]
    \draw (1,7) -- (2,7); 
    \draw (3,7) -- (4,7); 
    \draw (0,7) arc(180:0:2.5cm and 1cm);
    \draw (2,7) arc(-180:0:.5cm);
    \draw (4,7) arc(-180:0:.5cm);

    \draw (0,7) -- (0,5.1); \scc{0.1}{5}{3} \draw (0.1,5) -- (2,5);
    \draw (1,7) -- (1,6.1); \scc{1.1}{6}{3} \draw (1.1,6) -- (2,6);
    \draw[red] (0,6.5) -- (1,6.5);\draw (0.5,6.7) node {\scalebox{.5}{$x$}};

    \foreach \y in {1, ..., 6}{\smcr{2}{\y}\smcr{4}{\y}%
     \draw (3,\y) -- (4,\y);
     \draw (5,\y) -- (5.4+.2*\y,\y);}
    

    \sqp{5.7}{1}{-.1}{5}{0}
    \sqp{5.8}{2}{-.2}{2}{0} 
    \sqp{6.1}{3}{-.5}{1.7}{.9}
    \sqp{6.2}{4}{-.6}{1.6}{1.9}
    \sqp{6.5}{5}{-.9}{1.3}{2.9}
    \sqp{6.6}{6}{-1}{1.2}{3.9}


   \foreach \y in {1,3}{ 
   \scc{2-.2*\y-.1}{\y-.1}{2}\draw (2-.2*\y-.1,\y)--(2,\y); }
   \foreach \y in {2,4}{ 
   \scc{2-.2*\y}{\y-.1}{2}\draw (2-.2*\y,\y)--(2,\y); }

    \draw (2,0) arc(180:0:.5cm);
    \draw (4,0) arc(180:0:.5cm); 
    \draw (3,0) -- (4,0); 

   \end{scope}

   \begin{scope}[xshift=5cm, scale=.5, every node/.style = vert]
   \begin{scope}[draw opacity=0.1]
    \draw (1,7) -- (2,7); 
    \draw (3,7) -- (4,7); 
    \draw (0,7) arc(180:0:2.5cm and 1cm);
    \draw (2,7) arc(-180:0:.5cm);
    \draw (4,7) arc(-180:0:.5cm);

    \draw (0,7) -- (0,5.1); \scc{0.1}{5}{3} \draw (0.1,5) -- (2,5);
    \draw (1,7) -- (1,6.1); \scc{1.1}{6}{3} \draw (1.1,6) -- (2,6);
    \draw[red] (0,6.5) -- (1,6.5);  

    \foreach \y in {1, ..., 6}{\smcr{2}{\y}\smcr{4}{\y}%
     \draw (3,\y) -- (4,\y);
     \draw (5,\y) -- (5.4+.2*\y,\y);}
    

    \sqp{5.7}{1}{-.1}{5}{0}
    \sqp{5.8}{2}{-.2}{2}{0} 
    \sqp{6.1}{3}{-.4}{1.8}{.9}
    \sqp{6.2}{4}{-.6}{1.6}{1.9}
    \sqp{6.5}{5}{-.8}{1.4}{2.9}
    \sqp{6.6}{6}{-1}{1.2}{3.9}


   \foreach \y in {1, ..., 4}{ \scc{2-.2*\y}{\y-.1}{2}\draw (2-.2*\y,\y)--(2,\y); }

    \draw (2,0) arc(180:0:.5cm);
    \draw (4,0) arc(180:0:.5cm); 
    \draw (3,0) -- (4,0); 
   \end{scope}
   \begin{scope}[xshift = 1cm, yshift = -.5cm]   
    \Verts{a7/1.5/7, d7/3.5/7, x/-.5/7}          
    \draw (-1,7) node[empty] {\scalebox{0.5}{$x$}};
    \draw (4,7) node[empty] {\scalebox{0.5}{$z$}};
    \draw (2,7) node[empty] {\scalebox{0.5}{$z'$}};
    \draw (x) edge[bend  left] (a7);                
    \draw (x) edge[bend  left] (d7);  
    \foreach \l in {1,...,6}{                       
       \Verts{a\l/1.5/\l, b\l/1/\l+.5, c\l/2/\l+.5, d\l/3.5/\l, e\l/3/\l+.5, f\l/4/\l+.5 }
       \Edges{a\l/c\l, b\l/c\l, d\l/f\l, e\l/f\l }  
    }
    \foreach \l[evaluate=\l as  \lp using (int(\l+1))]in {1,...,6}{
       \Edges{a\lp/b\l,d\lp/e\l}  
       \draw (a\l) to[out = 0, in = -135] (2.5,\l+.5) to[out = 45, in = 180] (d\lp);
}
    \end{scope}
   \end{scope}

    \begin{scope}[xshift=10cm, scale=.5, yshift=-4cm, every node/.style = vert]
           \Verts{a7/1.5/7, d7/3.5/7, x/-.5/7, a1/2/6.5, a2/1/6.5, a3/1.5/6}
           \draw (x) edge[bend  left] (a7);                
           \draw (x) edge[bend  left] (d7);  
           \Edges{a7/a2, a1/a2, a1/a3};
           \draw (a3) to[out = 0, in = -135] (2.5,6.5) to[out = 45, in = 180] (d7); 
    \draw (-1,7) node[empty] {\scalebox{0.5}{$x$}};
    \draw (4,7) node[empty] {\scalebox{0.5}{$z$}};
    \draw (2,7) node[empty] {\scalebox{0.5}{$z'$}}; 
    \end{scope}

   \end{tikzpicture} 
   \caption{The smoothed link diagram and Lando graph $G_{\widetilde{D}_P}$ of $\widetilde{D}_{P}$
       and its subgraph $G^*$ when $p = q = r=7$}\label{A2} 
    \end{figure}
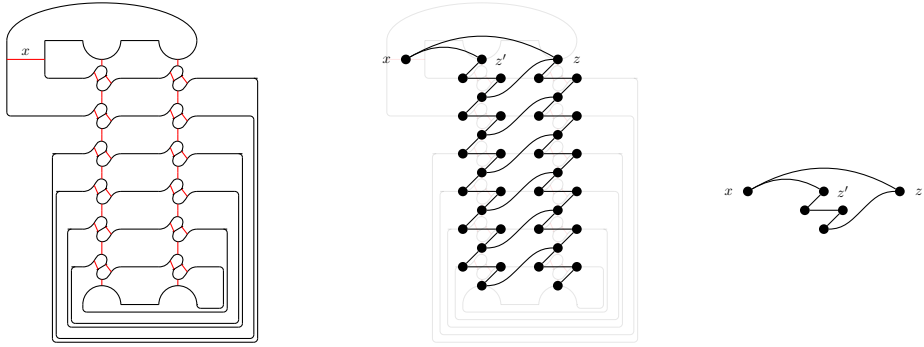

Figure \ref{A2} shows the smoothing $s_{A}\widetilde{D}_{P}$ of $\widetilde{D}_{P}$, and its Lando graph $G_{\widetilde{D}_P}$. We have drawn the graph so that the vertices correspond to the $A$-chords of the smoothed link diagram that are in approximately the same position. We note that $G_{\widetilde{D}_P}$ can be constructed from $G^*$ (in this case, $G^*$ is a $6$-cycle), by successively attaching $p-2$ copies of an $8$-cycle by identifying edges, and then attaching a $3$-path; we refer to this as a `tail of length $p-2$'.   

           
Recall that by Proposition \ref{(basic)} the independence complex of a $6$-cycle is $S^1 \vee S^1$.  When we attach a $8$-cycle by identifying an edge, by Corollary \ref{(glue)} we take the suspension of the independence complex twice, and when we attach a $3$-path, by part (2) of Proposition \ref{(divid)} we take a suspension of the independence complex.  Thus, the independence complex of $G_{\widetilde{D}_P}$ is $S^1 \vee S^1$ suspended $2(p-2) + 1 = 2(q-2) + 1$ times; that is 
\[ I(G_{\widetilde{D}_P}) \simeq \Sigma^{2q-3} (S^1 \vee S^1) \simeq S^{2q-2} \vee S^{2q-2},\]
 as needed.

Observe that depending on the values $p,q,$ and $r$, we could have `used' more or fewer than the $m = p-1$ crossings on the leftmost strand of the pretzel that we used in making $\widetilde{D}_P$. In the general case we will use $m = \min\{p,q-1,r-1\}$. Figure \ref{A3} shows that the case $(p,q,r) = (8,7,9)$ differs from the base case $(p,q,r) = (7,7,7)$ when we use the same number $m = 6$ of the crossings on the leftmost strand of the pretzel. There are $2$ crossings remaining on the leftmost strand. There are also $2$ crossings remaining above the crossing corresponding to the vertex $z$ on the rightmost strand. We apply the 2nd Reidemeister move in the same way as in Figure \ref{(chamix)}(ii) to each of those 2 crossings on the rightmost strand. This corresponds, in the Lando graph $G_{\widetilde{D}_P}$, to having $2$ copies of the vertex $x$, and two small trees attached to the vertex $z$.  We call these `small-trees'.  The two leaves of each small-tree have the same neighbourhood, so by part (1) of Lemma \ref{(DL)} we can remove one the leaves without changing homotopy type of $I(G_{\widetilde{D}_P})$. We henceforth may assume that a small-tree is just a $2$-path.    

In general, when we use $m$ crossings on the leftmost strand of the pretzel to make $\widetilde{D}_P$ we leave $p - m$ crossings on the leftmost strand, $q - m$ and $r - m$ crossings above the concentric circles on the middle and rightmost strands, respectively.
In the Lando graph $G_{\widetilde{D}_P}$, this corresponds to having a tail of length $m-1$, $p-m$ copies of $x$, and $q-m-1$ and $r-m-1$ small-trees attached to the vertices $z'$ and $z$ respectively.    
 
For each case we will use some $m$ crossings in the leftmost strand of the pretzel. The Lando graph $G_{\widetilde{D}_P}$ will be some variation of $G^*$, plus a tail of length $m-1$ and a $3$-path, so the independence complex $I(G_{\widetilde{D}_P}))$ of $G_{\widetilde{D}_P}$ will be homotopy equivalent to $\Sigma^{2m-1} I^*$, where $I^*$ is the independence complex of the subgraph $G^*$ of $G_{\widetilde{D}_P}$.  

 The cases:\\
  
\noindent (1) When $p \geq q = r,$ we use $m = q-1$ crossings on the leftmost strand. Then the graph $G^*$ consists of a $6$-cycle with $p - m$ copies of the vertex $x$.  By part (1) of Lemma \ref{(DL)} all but one of these copies can be removed from $G^*$ without changing the homotopy type of $I^*$, so we may assume that $G^*$ is a $6$-cycle, so $I^* \simeq S^1 \vee S^1$ and again  
          \[ I(G_{\widetilde{D}_P}) \simeq \Sigma^{2m-1} (S^1 \vee S^1) \simeq \Sigma^{2q-3} (S^1 \vee S^1) = S^{2q-2} \vee S^{2q-2},\]
    as needed. \\


\noindent (2) When $p + 1 = q = r,$ we use all $m = p$ crossings on the leftmost strand, so $G^*$ has no copies of $x$, no small-trees attached to $z$ or $z'$.  As the $G^*$ is a $4$-path, we get $I^* \simeq S^1$ by  Proposition \ref{(basic)}, and so $I(G_{\widetilde{D}_P}) \simeq \Sigma^{2m-1}S^1 \simeq \Sigma^{2q-3}S^1 = S^{2q-2} = S^{q+r-2}$.  \\


\noindent (3) When $q \leq p$ and $q < r,$ we use $m = q-1$ crossings on the leftmost strand leaving  $p-q+1$ copies of $x$, no small-trees on $z'$, and $r-q$ small-trees on $z$.  By Proposition \ref{(DL)} we can remove all but one copies of $x$ and, by Lemma \ref{(DL)}(1) as the leaf of the small-trees are dominated by $z$, we can remove $z$, without changing the homotopy type of $I^*$. Thus $G^*$ is a $4$-path (the $6$-cycle with $z$ removed) and $r-q$ independent edges. The middle vertex of the $4$-path dominates the ends of the path, so we can remove it, leaving $r - q + 2$ independent edges, so $I^* \simeq S^{r-q+1}$, and as $m = q-1$, $I(G_{\widetilde{D}_P}) \simeq \Sigma^{2q-3}S^{r-q+1} \simeq S^{q + r - 2}$, as needed. \\

\noindent (4) This case is essentially the same as case (3); switching $q$ and $r$ just switches 
the small trees from $z$ to $z'$ in $G^*$.  \\

   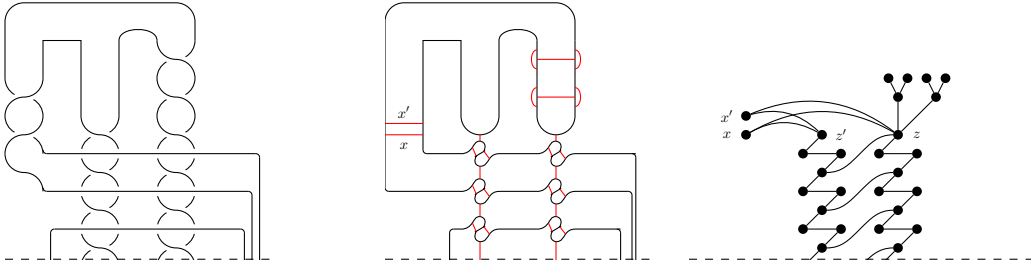
\begin{figure}
   \begin{tikzpicture}

   \begin{scope}[xshift = 0, scale=.5]
    \clipbox{0}{0.2}{7}{7}
    \draw (1,5) to (1,6) to  (2,6) to (2,4); 
    \draw (0,5) to (0,6.5) to[out = 90, in = 180] (.5,7) to (4.5,7) to[out = 0, in = 90] (5,6.5) to (5,6);
    \draw (3,4) to (3,6) to[out = 90, in = 90] (4,6);  
    
    \pCrr02\pCh034
    \nCh203\nCh405
    \old{1.2}{1}{6.2}
    \sqp{6.3}{1}{-.1}{1.6}{-.1}
    \ol{1}{2}{6.4}
    \sqp{6.5}{2}{-.1}{1.4}{-.1}
    \ol{1}{3}{6.6}
    \sqp{6.7}{3}{-.1}{1.2}{0.9}
   \end{scope}

   \begin{scope}[xshift=5cm, scale=.5]
    \clipbox{0}{0.2}{7}{7}
    \draw (1,4) to (1,6) to  (2,6) to (2,4); 
    \draw (0,4) to (0,6.5) to[out = 90, in = 180] (.5,7) to (4.5,7) to[out = 0, in = 90] (5,6.5) to (5,6);
    \draw (3,4) to (3,6) to[out = 90, in = 90] (4,6);  
    \draw (2,4) arc (180:360:.5); 
    \draw (4,4) arc (180:360:.5);

    \nsj{4}{4}\nsj{4}{5}

    \draw (0,4) -- (0,2.1); \scc{0.1}{2}{3} \draw (0.1,2) -- (2,2);
    \draw (1,4) -- (1,3.1); \scc{1.1}{3}{3} \draw (1.1,3) -- (2,3);
    \draw[red] (0,3.5) -- (1,3.5);\draw (0.5,3.2) node {\scalebox{.5}{$x$}};
    \draw[red] (0,3.8) -- (1,3.8);\draw (0.5,4.1) node {\scalebox{.5}{$x'$}};

    \foreach \y in {1, ..., 3}{\smcr{2}{\y}\smcr{4}{\y}%
     \draw (3,\y) -- (4,\y);
     \draw (5,\y) -- (6+.2*\y,\y);}
     \draw[red]  (2.5,.5) -- (2.5,0);
     \draw[red]  (4.5,.5) -- (4.5,0);
    

    \sqp{6.2}{1}{-.1}{1.6}{-.1}
    \sqp{6.5}{2}{-.1}{1.3}{-.1}
    \sqp{6.6}{3}{-.1}{1.7}{0.9}


     \scc{1.7}{.9}{2}\draw (1.8,1)--(2,1); 

   \end{scope}

   \begin{scope}[xshift=10cm, scale=.5, yshift=-3.5cm,  every node/.style = vert]
     \clipbox{-2}{3.7}{7}{10.7}

    \Verts{a7/1.5/7, d7/3.5/7, x/-.5/7, x'/-.5/7.5}          
    \draw (-1,7) node[empty] {\scalebox{0.5}{$x$}};
    \draw (-1,7.5) node[empty] {\scalebox{0.5}{$x'$}};
    \draw (4,7) node[empty] {\scalebox{0.5}{$z$}};
    \draw (2,7) node[empty] {\scalebox{0.5}{$z'$}};
    \draw (x) edge[bend  left] (a7);                
    \draw (x) edge[bend  left] (d7);                
    \draw (x') edge[bend  left] (a7);                
    \draw (x') edge[bend  left] (d7);  

    \Verts{x1/3.5/8, y1/4.5/8, x2/3.25/8.5, x3/3.75/8.5, y2/4.25/8.5, y3/4.75/8.5}
    \Edges{d7/x1, d7/y1, x1/x2, x1/x3, y1/y2, y1/y3}
    \foreach \l in {3,...,6}{                       
       \Verts{a\l/1.5/\l, b\l/1/\l+.5, c\l/2/\l+.5, d\l/3.5/\l, e\l/3/\l+.5, f\l/4/\l+.5 }
       \Edges{a\l/c\l, b\l/c\l, d\l/f\l, e\l/f\l }  
    }
    \foreach \l[evaluate=\l as  \lp using (int(\l+1))]in {3,...,6}{
       \Edges{a\lp/b\l,d\lp/e\l}  
       \draw (a\l) to[out = 0, in = -135] (2.5,\l+.5) to[out = 45, in = 180] (d\lp);
}
    
   \end{scope}
   \end{tikzpicture} 
   \caption{$\widetilde{D}_{P}, s_{A}\widetilde{D}_{P},$ and $G_{\widetilde{D}_P}$  when $(p,q,r) = (8,7,9)$}\label{A3} 
    \end{figure}  



\noindent (5) When $p<{\rm min}\{q,r\}-1,$ we use all $m = p$ crossings on the leftmost strand, so $G^*$ has no copies of $x$ and there are $q-p-1$ and $r-p-1$ small-trees attached at $z'$ and $z,$ respectively.  The vertices $z$ and $z'$ dominate the leaves of their small-trees, so we can remove them leaving $G^*$ as $(q + r - 2p - 2)$ independent edges and a $2$-path. The ends of a $2$-path have the same neighbourhoods, so we can remove one of them leaving another edge.  Therefore, $I^* \simeq S^{q+r-2p-2}$ and as $m = p$ we have 
           $I(G_{\widetilde{D}_P}) \simeq \Sigma^{2p-1} S^{q + r - 2p -2} \simeq S^{q + r - 3}$, as needed. 
\end{proof}

\begin{remark} \label{rem}
In the above proof, when both of the vertices $z$ and $z'$ have small-trees attached, and the vertex $x$ is there, the graph $G^*$ has an independent vertex and so its independence complex is contractible. We do not know how to deal with this case, and so were not able to deal with the cases where $\min\{q,r\}= p+1$ and $q \neq r$.
\end{remark}

It was shown in \cite{CS20} that if two link diagrams $D$ and $D'$ represent the same link, then the associated independence complex $I(G_{D'})$ is either contractible or has the same homotopy type as $I(G_{D})$. Therefore, Proposition \ref{main0}, Theorem \ref{main1}, and Theorem \ref{main2} directly imply the following:

\begin{corollary}

Let $(p,q,r)$ be a triple of positive integers.
\begin{enumerate}
    \item The real extreme Khovanov homology groups of the pretzel links $P(p,q,r)$, $P(-p,-q,-r)$, and $P(p,q,-r)$ are torsion-free.
    \item The real extreme Khovanov homology group of the pretzel link $P(p,-q,-r)$ is torsion-free under the condition that if $\min\{q,r\} = p+1,$ then $q = r$.
\end{enumerate}

\end{corollary}

This agrees with Conjecture \ref{conjecture}.

Furthermore, by applying the relationship between the extreme Khovanov homology and the simplicial homology of the associated independence complex of a given link diagram described in Theorem \ref{GMS main} to Proposition \ref{main0}, Theorem \ref{main1}, and Theorem \ref{main2}, the real-extreme Khovanov homology of certain pretzel links can be computed as follows:

\begin{corollary}
Suppose that $(p,q,r)$ is a triple of positive integers such that if $\min\{q,r\} = p+1,$ then $q = r.$ Then the real-extreme Khovanov homology of a pretzel link is
\begin{enumerate}
    \item $KH^{i,\underline{j}}(P(p, q, r)) \cong \mathbb{Z}$ when $i=-n, \underline{j}=p+q+r-3n-3;$
    \item $KH^{i,\underline{j}}(P(-p, -q, -r)) \cong \mathbb{Z}$ when $i=-n, \underline{j}=-3n+1;$
    \item $KH^{i,\underline{j}}(P(p, q, -r)) \cong \mathbb{Z}$ when $i=-n, \underline{j}=p+q-3n-1;$
    \item $KH^{i,\underline{j}}(P(p, -q, -r)) \cong
\begin{cases}
\mathbb{Z} \oplus \mathbb{Z} {\rm~when~} i=q-n, \underline{j}=p+2q-3n-1 & \text{if } p \geq q=r;\\
\mathbb{Z} {\rm~when~} i=p-n+1, \underline{j}=3p-3n+1 & \text{if } p+1=q=r;\\
\mathbb{Z} {\rm~when~} i=q-n, \underline{j}=p-3n-1+2q & \text{if } q \leq p {\rm~and~} q<r;\\
\mathbb{Z} {\rm~when~} i=r-n, \underline{j}=p-3n-1+2r & \text{if } r \leq p {\rm~and~} r<q;\\
\mathbb{Z} {\rm~when~} i=p-n, \underline{j}=3p-3n+1 & \text{if } p<{\rm min}\{q,r\}-1,\\
\end{cases}$
\end{enumerate}
where $n$ denotes the number of negative crossings in the standard link diagram of the pretzel link.

\end{corollary}

\begin{proof}

\noindent (1) Let $P$ be the pretzel link $P(p, q, r)$, and $D_P$ be the standard link diagram of $P.$ Then the corresponding Lando graph $G_{D_P}$ is the empty graph, i.e., $I(G_{D_P}) \simeq S^{-1}.$ \\
By Theorem \ref{GMS main}, we have 
\begin{align*} 
KH^{i,j_{\rm min}}(D_P) & \cong {\widetilde{H}}^{i-1+n}({{I}({G_{D_P}})})\\ 
    	              & \cong {\widetilde{H}}^{i-1+n}(S^{-1})\\ 
                        & \cong \begin{cases}
\mathbb{Z} & \text{if } i=-n;\\
0 & \text{otherwise.}\\
\end{cases}
\end{align*} 
Since $c=p+q+r$ and $|s_{A}D_P|=3,$ $j_{\rm min}= p+q+r-3n-3$ by Corollary \ref{GMS cor}.\\
Thus, $KH^{i,\underline{j}}(P(p, q, r)) \cong \mathbb{Z}$ when $i=-n, \underline{j}=p+q+r-3n-3,$ as desired.
\\

\noindent (2) 
Let $D_P$ be the standard link diagram of $P=P(-p, -q, -r).$ Then again $G_{D_P}$ is the empty graph, i.e., $I(G_{D_P}) \simeq S^{-1}.$\\
By Theorem \ref{GMS main},
\begin{align*} 
KH^{i,j_{\rm min}}(D_P) & \cong {\widetilde{H}}^{i-1+n}(S^{-1})\\ 
                        & \cong \begin{cases}
\mathbb{Z} & \text{if } i=-n;\\
0 & \text{otherwise.}\\
\end{cases}
\end{align*} 
Since $c=p+q+r$ and $|s_{A}D_P|=p+q+r-1,$ $j_{\rm min}= p+q+r-3n-(p+q+r-1)=-3n+1$ by Corollary \ref{GMS cor}.\\
Hence, $KH^{i,\underline{j}}(P(-p, -q, -r)) \cong \mathbb{Z}$ when $i=-n, \underline{j}=-3n+1.$
\\

\noindent (3) 
Let $\widetilde{D}_{P}$ denote the link diagram of $P(p, q, -r)$ obtained in Theorem \ref{main1}. 
Let $\widetilde{c}$ and $\widetilde{n}$ be the number of crossings and the number of negative crossings of $\widetilde{D}_{P}$, respectively.
Note that $\widetilde{D}_{P}$ contains $r$ additional negative crossings compared to the standard link diagram of $P(p, q, -r)$ and $I(G_{\widetilde{D}_{P}}) \simeq S^{r-1}.$
Thus, by Theorem \ref{GMS main}, we have 
\begin{align*} 
KH^{i,j_{\rm min}}(\widetilde{D}_{P}) & \cong {\widetilde{H}}^{i-1+\widetilde{n}}({{I}({G_{\widetilde{D}_{P}}})})\\ 
    	              & \cong {\widetilde{H}}^{i-1+n+r}(S^{r-1})\\ 
                        & \cong \begin{cases}
\mathbb{Z} & \text{if } i= -n;\\
0 & \text{otherwise.}\\
\end{cases}
\end{align*} 
Since $\widetilde{c}=p+2+q+r+2(r-1)$ and $|s_{A}\widetilde{D}_P|=1,$ $j_{\rm min}=\widetilde{c} - 3\widetilde{n} - |s_{A}\widetilde{D}_P| = \{p+2+q+r+2(r-1)\} -3(n+r)-1= p+q-3n-1$ by Corollary \ref{GMS cor}. \\
Thus, $KH^{i,\underline{j}}(P(p, q, -r)) \cong \mathbb{Z}$ when $i=-n, \underline{j}=p+q-3n-1.$
\\

\noindent (4) 
Let $\widetilde{D}_{P}$ denote the link diagram of $P(p, -q, -r)$ obtained in Theorem \ref{main2}. 
Let $\widetilde{c}$ and $\widetilde{n}$ be the number of crossings and the number of negative crossings of $\widetilde{D}_{P}$, respectively.
Observe that when we undo the bottom crossing in the leftmost strand of the standard link diagram of $P(p, -q, -r)$ by crossing it over the other two strands, this process increases the number of crossings by three, including one negative crossing. To be more specific, as shown in Figure \ref{pretzel}, since $u_i$'s in $D_P$ and $\widetilde{u}_i$'s in $\widetilde{D}_{P}$ are essentially the same crossing, three additional crossings appear in each process: one crossing indicated by a dotted circle and two shaded crossings. The crossings indicated by the dotted circles are self-crossings and their signs are always positive. The two shaded crossings on the same level have opposite signs, that is, one is positive and the other is negative.

\begin{figure}[ht!]
\centering
\includegraphics[width=90mm]{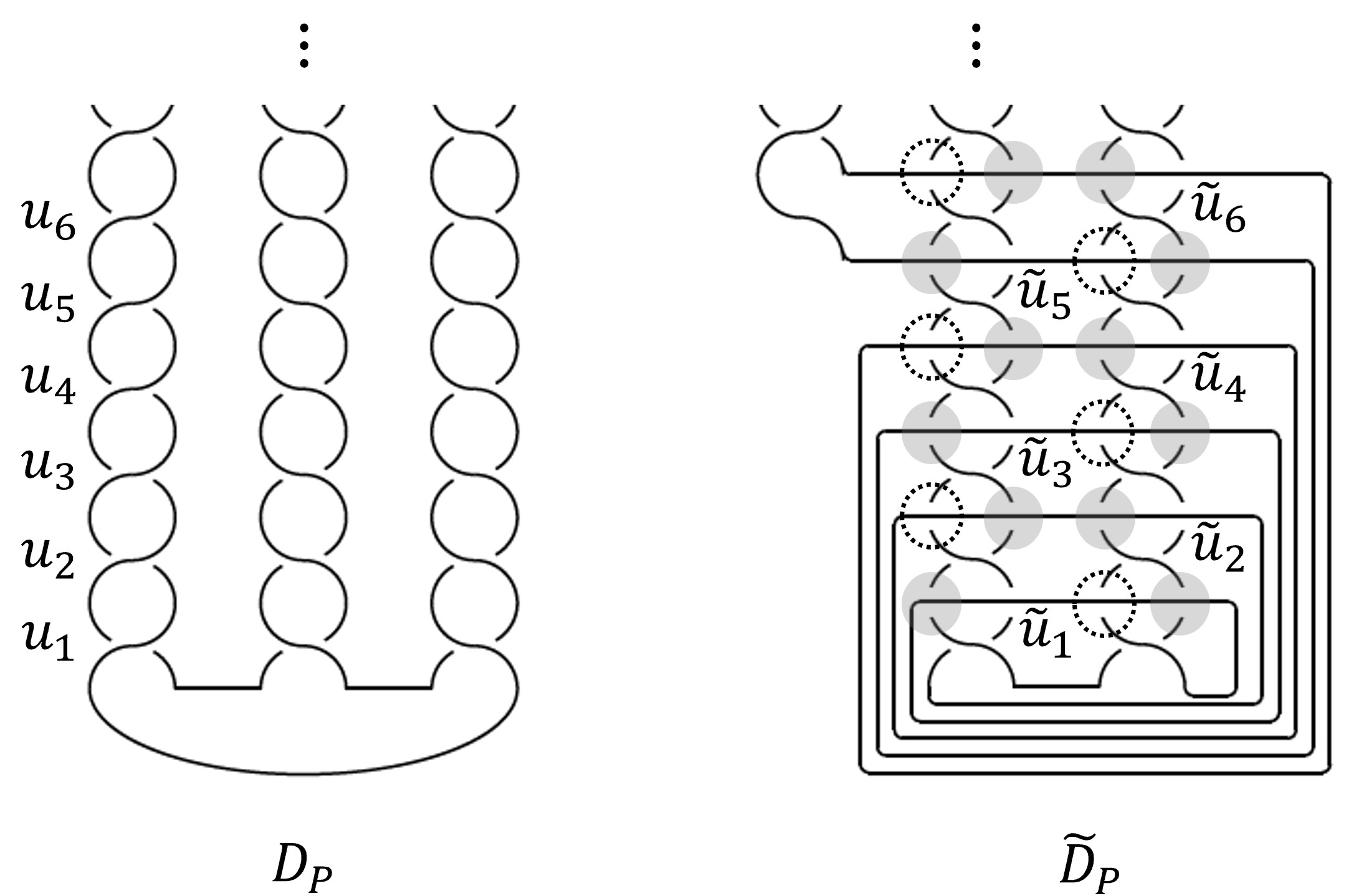}
\caption {The changes in the number of crossings when transforming ${D}_{P}$ to $\widetilde{D}_{P}$}\label{pretzel}
\end{figure}

\begin{enumerate}[label=\roman*.]
    \item When $p \geq q = r,$ to deform the standard link diagram of $P(p, -q, -r)$ into $\widetilde{D}_{P},$ we use $q-1$ crossings on the leftmost strand. Then $\widetilde{D}_{P}$ contains $3q-3$ additional crossings, of which $q-1$ are negative crossings, compared to the standard link diagram of $P(p, -q, -r).$ Note that $I(G_{\widetilde{D}_{P}}) \simeq S^{2q-2} \vee S^{2q-2}$ by Theorem \ref{main2}. 
    Thus, by Theorem \ref{GMS main}, we have 
\begin{align*} 
KH^{i,j_{\rm min}}(\widetilde{D}_{P}) & \cong {\widetilde{H}}^{i-1+\widetilde{n}}({{I}({G_{\widetilde{D}_{P}}})})\\ 
    	              & \cong {\widetilde{H}}^{i-1+n+q-1}(S^{2q-2} \vee S^{2q-2})\\ 
                        & \cong \begin{cases}
\mathbb{Z} \oplus \mathbb{Z} & \text{if } i= q-n;\\
0 & \text{otherwise.}\\
\end{cases}
\end{align*}
Since $\widetilde{c}=p+q+r+3q-3$ and $|s_{A}\widetilde{D}_P|=1,$ $j_{\rm min}=\widetilde{c} - 3\widetilde{n} - |s_{A}\widetilde{D}_P|=( p+q+r+3q-3) -3(n+q-1)-1= p+q+r-3n-1$ by Corollary \ref{GMS cor}. \\
Thus, $KH^{i,\underline{j}}(P(p, q, -r)) \cong \mathbb{Z} \oplus \mathbb{Z}$ when $i=q-n$ and $\underline{j}=p+2q-3n-1.$\\

    \item When $p + 1 = q = r,$ we use all $p$ crossings on the leftmost strand of the standard link diagram. Then $\widetilde{D}_{P}$ contains $3p$ additional crossings, of which $p$ are negative crossings, compared to the standard link diagram of $P(p, -q, -r).$ Note that $I(G_{\widetilde{D}_{P}}) \simeq S^{q+r-2}$ by Theorem \ref{main2}. By Theorem \ref{GMS main},  
\begin{align*} 
KH^{i,j_{\rm min}}(\widetilde{D}_{P}) & \cong {\widetilde{H}}^{i-1+\widetilde{n}}({{I}({G_{\widetilde{D}_{P}}})})\\ 
    	              & \cong {\widetilde{H}}^{i-1+n+p}(S^{q+r-2})\\ 
                        & \cong \begin{cases}
\mathbb{Z}  & \text{if } i= p+1-n;\\
0 & \text{otherwise.}\\
\end{cases}
\end{align*}
Since $\widetilde{c}=p+q+r+3p$ and $|s_{A}\widetilde{D}_P|=1,$ $j_{\rm min}= (p+q+r+3p) -3(n+p)-1= p+q+r-3n-1$ by Corollary \ref{GMS cor}. \\
Thus, $KH^{i,\underline{j}}(P(p, q, -r)) \cong \mathbb{Z}$ when $i=p+1-n$ and $\underline{j}=3p-3n+1.$\\

    \item When $q \leq p$ and $q < r,$ we use $q-1$ crossings on the leftmost strand of the standard link diagram. Furthermore, we apply the 2nd Reidemeister move to each of the $r-q$ crossings above the concentric circles on the rightmost strand.    
    Thus, $\widetilde{D}_{P}$ contains $3(q-1)+2(r-q)$ additional crossings, of which $(q-1)+(r-q)$ are negative crossings, compared to the standard link diagram of $P(p, -q, -r).$ Note that $I(G_{\widetilde{D}_{P}}) \simeq S^{q+r-2}$ by Theorem \ref{main2}. By Theorem \ref{GMS main},  
\begin{align*} 
KH^{i,j_{\rm min}}(\widetilde{D}_{P}) & \cong {\widetilde{H}}^{i-1+\widetilde{n}}({{I}({G_{\widetilde{D}_{P}}})})\\ 
    	              & \cong {\widetilde{H}}^{i+n+r-2}(S^{q+r-2})\\ 
                        & \cong \begin{cases}
\mathbb{Z}  & \text{if } i= q-n;\\
0 & \text{otherwise.}\\
\end{cases}
\end{align*}
Since $\widetilde{c}=p+q+r+3(q-1)+2(r-q)$ and $|s_{A}\widetilde{D}_P|=1,$ $j_{\rm min}= \{p+q+r+3(q-1)+2(r-q)\} -3\{n+(q-1)+(r-q)\}-1= p+2q-3n-1$ by Corollary \ref{GMS cor}. \\
Thus, $KH^{i,\underline{j}}(P(p, q, -r)) \cong \mathbb{Z}$ when $i=q-n$ and $\underline{j}=p+2q-3n-1.$\\

    \item This case is essentially the same as Case iii; we only need to switch $q$ and $r.$\\

    \item When $p<{\rm min}\{q,r\}-1,$ we use all $p$ crossings on the leftmost strand of the standard link diagram. Furthermore, we apply the 2nd Reidemeister move to each of the $q-p-1$ and $r-p-1$ crossings above the concentric circles on the middle and rightmost strands, respectively.
    Thus, $\widetilde{D}_{P}$ contains $3p+2(q-p-1)+2(r-p-1)$ additional crossings, of which $p+(q-p-1)+(r-p-1)$ are negative crossings, compared to the standard link diagram of $P(p, -q, -r).$ Note that $I(G_{\widetilde{D}_{P}}) \simeq S^{q+r-3}$ by Theorem \ref{main2}. Then by Theorem \ref{GMS main}, we have
\begin{align*} 
KH^{i,j_{\rm min}}(\widetilde{D}_{P}) & \cong {\widetilde{H}}^{i-1+\widetilde{n}}({{I}({G_{\widetilde{D}_{P}}})})\\ 
    	              & \cong {\widetilde{H}}^{i+n+q+r-p-3}(S^{q+r-3})\\ 
                        & \cong \begin{cases}
\mathbb{Z}  & \text{if } i= p-n;\\
0 & \text{otherwise.}\\
\end{cases}
\end{align*}
Since $\widetilde{c}=p+q+r+3p+2(q-p-1)+2(r-p-1)$ and $|s_{A}\widetilde{D}_P|=1,$ $j_{\rm min}=\{p+q+r+3p+2(q-p-1)+2(r-p-1)\} -3\{n+p+(q-p-1)+(r-p-1)\}-1= 3p-3n+1$ by Corollary \ref{GMS cor}. \\
Therefore, $KH^{i,\underline{j}}(P(p, q, -r)) \cong \mathbb{Z}$ when $i=p-n$ and $\underline{j}=3p-3n+1,$ as desired.
\end{enumerate}

\end{proof}

\section*{Acknowledgement}
The paper grew out of the master thesis of the fourth author at the Kyungpook National University. The work of Seung Yeop Yang was supported by the National Research Foundation of Korea (NRF) grant funded by the Korean government (MSIT)(No. 2022R1A5A1033624) and by Global - Learning \& Academic research institution for Master’s·PhD students, and Postdocs (LAMP) Program of the National Research Foundation of Korea (NRF) grant funded by the Ministry of Education (No. RS-2023-00301914). The work of Mark Siggers was supported by a National Research Foundation of Korea (NRF-2022R1A2C1091566) grant funded by the Ministry of Education.

\end{document}